\documentclass[12pt]{article}
\usepackage{graphicx}
\usepackage{amsmath,amsfonts,amssymb,amsthm}
\usepackage{bbm,mathrsfs,bm,mathtools}
\title{The Gauss image problem for pseudo-cones}
\author{Rolf Schneider}
\date{}
\date{}
\newcommand{\Sn}{{\mathbb S}^{n-1}}

\newcommand{\R}{{\mathbb R}}
\newcommand{\C}{{\mathcal C}}
\newcommand{\K}{{\mathcal K}}

\newcommand{\N}{{\mathbb N}}

  \newcommand{\D}{{\rm d}}

  \newcommand{\fed}{\,\rule{.1mm}{.20cm}\rule{.20cm}{.1mm}\,}

\newtheorem{theorem}{Theorem}
\newtheorem{lemma}{Lemma}
\newtheorem{definition}{Definition}

\begin{document}
\maketitle

\begin{abstract}
{The Gauss image problem for convex bodies asks for the existence of a convex body that ``links'' two given measures on the unit sphere in a certain way. We treat here a corresponding question for pseudo-cones, that is, for unbounded closed convex sets strictly contained in their recession cones.}\\[2mm]
\end{abstract}

\section{Introduction}\label{sec1}

The Gauss image problem, as suggested by B\"or\"oczky, Lutwak, Yang, Zhang, Zhao \cite{BLYZZ20}, can be described as follows. In Euclidean space $\R^n$ ($n\ge 2$) with unit sphere $\Sn$, let $K$ be a compact convex set containing the origin $o$ in the interior. Define the radial map $r_K:\Sn\to\partial K$ such that $r_K(u)=ru\in\partial K$ with $r>0$. For $\eta\subset\Sn$, the radial Gauss image ${\boldsymbol\alpha}_K(\eta)$ is defined as the set of all outer unit normal vectors of $K$ at points of $r_K(\eta)$. If $\eta$ is a Borel set, then ${\boldsymbol\alpha}_K(\eta)$ is known (\cite{Sch14}, p. 88) to be (spherically) Lebesgue measurable. Let $\lambda$ be a finite measure on the Lebesgue measurable subsets of $\Sn$ which is absolutely continuous with respect to spherical Lebesgue measure. The Gauss image measure of $\lambda$, associated with $K$, is defined by
$$ \lambda(K,\eta):= \lambda({\boldsymbol\alpha}_K(\eta)), \quad \eta\subset\Sn \mbox{ Borel}.$$
It is known (\cite[Sect. 3]{BLYZZ20}) that $\lambda(K,\cdot)$ is a measure. The Gauss image problem now asks: Given a finite Borel measure $\mu$ on $\Sn$, what are the necessary and sufficient conditions on $\lambda$ and $\mu$ in order that there is a convex body $K$ (with $o\in{\rm int}\,K$) such that $\lambda(K,\cdot) =\mu$ on the Borel subsets of $\Sn$? 

When $\lambda$ is spherical Lebesgue measure, then $\lambda(K,\cdot)$ is known as Aleksandrov's integral curvature; in this case, the problem was already solved by Aleksandrov \cite{Ale39}. The general Gauss image problem was proposed in \cite{BLYZZ20}. In that article, it was proved that for the existence of $K$ with $\lambda(K,\cdot) =\mu$ (where $\lambda$ and $\mu$ are non-zero) it is sufficient that the measures $\lambda$ and $\mu$ are Aleksandrov related (for the explanation of which we refer to \cite{BLYZZ20}). If $\lambda$ is positive on nonempty open sets, these authors also showed that the latter condition is necessary, and that $K$, if it exists, is uniquely determined up to a dilatation. If $\mu$ is discrete, Semenov \cite{Sem24} found a relaxation of the Aleksandrov condition that is necessary and sufficient. Semenov \cite{Sem22} also treated a variant of the Gauss image problem for convex bodies, where both measures are discrete. In \cite{Sem23} he studied the uniqueness question for the Gauss image problem in a general version (that is, without special assumptions on $\lambda$) and obtained several new properties of the (multi-valued) radial Gauss image map.

The aim of the following is to treat an analogue of the Gauss image problem for pseudo-cones. These sets can be considered, under several aspects, as a counterpart to the convex bodies containing the origin in the interior. By definition, a pseudo-cone in $\R^n$ is a nonempty closed convex subset $K\subset\R^n$ not containing the origin and such that $\lambda K\subseteq K$ for $\lambda\ge 1$ (equivalently, a closed convex set not containing the origin and contained in its recession cone; see \cite[Thm. 3.2]{XLL23}). Thus, such a set is unbounded. Its recession cone, a closed convex cone, is denoted by $C$, and we assume in the following that it is always pointed and $n$-dimensional. Its dual cone is defined by $C^\circ=\{x\in\R^n: \langle x,y\rangle\le 0\;\forall y\in C\}$. We are interested in the pseudo-cones $K$ with given recession cone $C$, called $C$-pseudo-cones. The role that the unit sphere plays above is now split between two open subsets of the unit sphere, given by
$$ \Omega_C:= \Sn\cap {\rm int}\,C,\qquad \Omega_{C^\circ}:= \Sn\cap{\rm int}\,C^\circ$$
(or their closures, indicated by ${\rm cl}$).

Part of the following is restricted to a special class of pseudo-cones. The $C$-pseudo-cone $K$ is called {\em internal} if  the vectors of $\Omega_{C^\circ}$ are attained as outer normal vectors of $K$ only at points in the interior of $C$. If $K$ is internal, it suffices (for the treatment of $\lambda(K,\cdot)$) to define the radial Gauss image ${\boldsymbol\alpha}_K(\eta)$ only for sets $\eta\subset \Omega_C$, namely as the set of all outer unit normal vectors of $K$ at points $rv\in\partial K$, where $v\in\eta$ and $r\in (0,\infty)$. We have ${\boldsymbol\alpha}_K(\eta)\subset{\rm cl}\,\Omega_{C^\circ}$, since all outer unit normal vectors of $K$ belong to the closure of $\Omega_{C^\circ}$. 

We assume that $\lambda$ is a non-zero, finite measure on the Lebesgue measurable subsets of ${\rm cl}\,\Omega_{C^\circ}$, which is zero on Borel sets of Hausdorff dimension $n-2$ (an assumption slightly weaker than the absolute continuity demanded in \cite{BLYZZ20}). For Borel sets $\eta\subset \Omega_C$ we define $\lambda(K,\eta):= \lambda({\boldsymbol\alpha}_K(\eta))$. Let $\mu$ be a Borel measure on $\Omega_C$. The Gauss image problem for pseudo-cones asks: What are the necessary and sufficient conditions on $\lambda$ and $\mu$ in order that there is a $C$-pseudo-cone $K$ with $\lambda(K,\cdot)=\mu$?

We have already seen in \cite{Sch18, Sch21, Sch24a} that conditions on the measures appearing in Minkowski type problems for convex bodies, such as not being restricted to a great subsphere, a centroid condition for surface area measures or a subspace concentration condition for cone-volume measures, do not play a role in the case of pseudo-cones. On the other hand, new problems arise with pseudo-cones, since their surface area measures and cone-volume measures can be infinite. Since presently we assume that the measures $\lambda$ and $\mu$ are finite, the Gauss image problem for pseudo-cones turns out to be easier than for convex bodies. We shall prove the following result.

\begin{theorem}\label{T1.1}
Let $\lambda$ be a measure on the Lebesgue measurable subsets of ${\rm cl}\,\Omega_{C^\circ}$ which is non-zero, finite, and zero on Borel sets of Hausdorff dimension $n-2$. Let $\mu$ be a Borel measure on $\Omega_C$. There exists a $C$-pseudo-cone $K$ with
$$ \mu =\lambda(K,\cdot)$$
if and only if $\lambda(\Omega_{C^\circ})=\mu(\Omega_C)$.
\end{theorem}

We have simplified here the Gauss image problem for pseudo-cones by allowing only measures $\mu$ on $\Omega_C$, and not on its closure. That this is indeed a restriction, can be seen from the following simple example. Let $K$ be the intersection of $C$ and a closed halfspace not containing the origin, which has an outer normal vector belonging to $\Omega_{C^\circ}$. Extending the definition of $\lambda(K,\cdot)$ to ${\rm cl}\,\Omega_C$, we see that this measure is concentrated on the boundary of $\Omega_C$. A generalization of Theorem \ref{T1.1} to measures on ${\rm cl}\,\Omega_C$ remains open. 

We shall prove Theorem \ref{T1.1} in Section \ref{sec5}, after some preparations. Section \ref{sec6} contains a uniqueness result.

The following has been pointed out to the author by Yiming Zhao. After applying two gnomonic projections, the Gauss image problem for $C$-pseudo-cones can be viewed as a measure transport problem in a Euclidean space of dimension $n-1$. Then a result of McCann \cite{McC95} (refining an earlier result of Brenier) comes to mind, according to which the transport can be achieved by the gradient map of a suitable convex function. However, in the setting of the current work, the sought-for transport map should be a {\em normalized} version of the gradient map. There seems to be no obvious way to directly apply the result of McCann.

As mentioned, the Gauss image problem for convex bodies reduces to Aleksandrov's integral curvature problem if one considers Lebesgue measure. Oliker \cite{Oli07} pointed out that this problem is connected to optimal mass transport, and he showed how his solution of Aleksandrov's problem implies the extremality of a certain total cost. A new solution of Aleksandrov's problem by mass transportation methods was given by Bertrand \cite{Ber16}. It appears conceivable that his proof can be extended to more general measures, and can then be carried over to pseudo-cones. The following treatment of the Gauss image problem for pseudo-cones seems to be more elementary.

\section{Fundamentals about pseudo-cones}\label{sec2}

This section fixes the notation and collects some known results. We recall that $C\subset\R^n$ is a closed convex cone, pointed and with interior points. Then its dual cone, $C^\circ$, has the same properties. A $C$-pseudo-cone $K$ is a closed convex set with $o\notin K$ and $\lambda K\subseteq K$ for $\lambda\ge 1$, with recession cone $C$. Necessarily, 
$$ K=C\cap\bigcap_{u\in \Omega_{C^\circ}} H_K^-(u),$$
where $H_K^-(u)$ is the supporting halfspace of $K$ with outer normal vector $u$. The set of all $C$-pseudo-cones is denoted by $ps(C)$, and $psi(C)$ denotes the set of all internal $C$-pseudo-cones.

If $\omega\subset\Omega_{C^\circ}$ is a nonempty compact subset, we say as in \cite[Sect. 8]{Sch18} that the $C$-pseudo-cone $K$ is $C$-determined by $\omega$ if
$$ K=C\cap\bigcap_{u\in \omega} H_K^-(u).$$
The set of all $C$-pseudo-cones which are $C$-determined by $\omega$ is denoted by $\K(C,\omega)$.

Convergence of pseudo-cones can be defined via the usual convergence of convex bodies. Let $B^n$ be the unit ball of $\R^n$ with center $o$. Let $K_1,K_2,\dots$ be a sequence in $ps(C)$. We say that this sequence converges to the $C$-pseudo-cone $K$ and write
$$ K_j\to K\quad\mbox{as }j\to \infty$$
if there exists $t_0>0$ such that $K_j\cap t_0B^n\not=\emptyset$ for $j\in\N$ and
$$ \lim_{j\to\infty} (K_j\cap tB^n)= K\cap tB^n\quad\mbox{for each }t\ge t_0.$$
This is, of course, equivalent to Definition 2 in \cite{Sch24}.

To emphasize the distinction between $\Omega_{C^\circ}$ and $\Omega_C$, we will mostly denote vectors and subsets of $\Omega_{C^\circ}$ by $u,\omega$, and vectors and subsets of $\Omega_C$ by $v,\eta$.

Let $K\in ps(C)$. We recall some notation from \cite{Sch24}. The support function of $K$ is defined by
$$ h_K(x):= \sup\{\langle x,y\rangle :y\in K\}\quad\mbox{for } x\in C^\circ.$$
Here sup can be replaced by max if $x\in{\rm int}\,C^\circ$. We have $h_K\le 0$, and therefore we also write $\overline h_K=-h_K$. 

We define the radial function of $K$ by
$$ \rho_K(v):= \min\{r>0: rv \in K\}\quad\mbox{for }v\in\Omega_C$$
(thus we define here the radial function only on $\Omega_C$). It is easy to see that $K$ is uniquely determined by its support function, as well as by its radial function.

The relations between support function and radial function are given by (1) and (2) in \cite{Sch24}, namely
\begin{equation}\label{2.1} 
\overline h_K(u) = \inf_{v\in\Omega_C}|\langle u,v\rangle|\rho_K(v)\quad\mbox{for }u\in{\rm cl}\,\Omega_{C^\circ}
\end{equation}
(where inf can be replaced by min if $u\in\Omega_{C^\circ}$) and
$$
\frac{1}{\rho_K(v)}= \min_{u\in {\rm cl}\,\Omega_{C^\circ}} \frac{|\langle v,u\rangle|}{\overline h_K(u)}\quad\mbox{for }v\in \Omega_C.
$$

The radial map $r_K:\Omega_C\to\partial K$ is defined by
$$ r_K(v):= \rho_K(v)v\quad\mbox{for }v\in \Omega_C.$$
For $\sigma\subseteq \partial K$, we denote by ${\boldsymbol \nu}_K(\sigma)$ the set of all outer unit normal vectors of $K$ at points of $\sigma$, and we repeat that the radial Gauss image of $\eta\subseteq \Omega_C$ is defined by
$$  {\boldsymbol\alpha}_K(\eta):= {\boldsymbol\nu}_K(r_K(\eta));$$
thus ${\boldsymbol\alpha}_K(\eta)\subseteq {\rm cl}\,\Omega_{C^\circ}$. 

By $\eta_K$ we denote the set of all $v\in\Omega_C$ for which the outer unit normal vector of $K$ at $r_K(v)$ is not unique. By $\omega_K\subset{\rm cl}\,\Omega_{C^\circ}$ we denote the set of all $u\in{\rm cl}\,\Omega_{C^\circ}$ which are outer unit normal vectors of $K$ at more than one point of $\partial K$. 
It is known that $\eta_K$ and $\omega_K$ have spherical Lebesgue measure zero. The set $\sigma_K:= r_K(\eta_K)$ has $(n-1)$-dimensional Hausdorff measure zero. About $\omega_K$, we are more precise. It follows from \cite[Thm. 2.2.5]{Sch14} and polarity that $\omega_K$ can be covered by countably many sets of finite $(n-2)$-dimensional Hausdorff measure, hence if we assume that the measure $\lambda$ is zero on sets of Hausdorff dimension $n-2$, then
\begin{equation}\label{2.x}
\lambda (\omega_K)=0.
\end{equation}

The Gauss map $\nu_K:\partial K\setminus\sigma_K\to \Omega_{C^\circ}$ of $K$ is defined by letting $\nu_K(y)$ be the unique outer unit normal vector of $K$ at $y\in\partial K \setminus \sigma_K$. The reverse spherical Gauss map $x_K:\Omega_{C^\circ}\setminus\omega_K\to\partial K$ is defined by letting $x_K(u)$ be the unique point of $\partial K$ at which $u\in \Omega_{C^\circ}\setminus\omega_K$ is attained as outer unit normal vector.  (Note that a vector in $\partial\Omega_{C^\circ}$ (boundary with respect to $\Sn$) may be attained as a normal vector of $K$, but never at a unique boundary point.)

For $K\in ps(C)$ we define the radial Gauss map 
$$ \alpha_K:\Omega_C\setminus\eta_K\to {\rm cl}\,\Omega_{C^\circ}\quad\mbox{by} \quad \alpha_K:= \nu_K\circ r_K.$$
We point out that the radial map and the radial Gauss map are only defined on (parts of) $\Omega_C$, not on its closure.
Modifying (2.17) in \cite{HLYZ16}, we define for $K\in psi(C)$ the reverse radial Gauss map
$$ \alpha_K^*: \Omega_{C^\circ}\setminus\omega_K\to \Omega_C\quad\mbox{by}\quad \alpha_K^*:= r_K^{-1}\circ x_K.$$
Thus, $\alpha_K^*$ is not defined on the boundary of $\Omega_{C^\circ}$. We have restricted this definition to $psi(C)$, in order that for $u\in\Omega_{C^\circ}\setminus\omega_K$ the vector $r_K^{-1}(x_K(u))$ belong to $\Omega_C$. The maps $\alpha_K$ and $\alpha_K^*$ are continuous.

On $\Omega_{C^\circ}\setminus\omega_K$ the inverse map $\alpha_K^{-1}$ is well defined, thus for $K\in psi(C)$ we have
$$ \alpha_K^{-1}=\alpha_K^*\quad\mbox{$\lambda$-almost everywhere on $\Omega_{C^\circ}$}.$$

We state that for $K\in psi(C)$
\begin{equation}\label{2.3}
{\boldsymbol\alpha}_K(\eta)\setminus\omega_K  = (\alpha_K^*)^{-1}(\eta) \quad\mbox{for }\eta\subset \Omega_C.
\end{equation}
For the proof, let $\eta\subset \Omega_C$ and $u\in\Omega_{C^\circ}\setminus\omega_K$. Then ${\bf x}_K(\{u\})=x_K(u)$ and hence
\begin{eqnarray*}
&& u\in {\boldsymbol \alpha}_K(\eta) = {\boldsymbol\nu}_K(r_K(\eta))\\
&& \Leftrightarrow x_K(u)\in r_K(\eta) \Leftrightarrow r_K^{-1}(x_K(u))\in\eta\Leftrightarrow \alpha^*_K(u)\in\eta\\
&& \Leftrightarrow u\in (\alpha_K^*)^{-1}(\eta).
\end{eqnarray*}

For a $C$-pseudo-cone $K\in ps(C)$, the copolar set is defined by
$$ K^*:= \{x\in\R^n: \langle x,y\rangle\le -1\,\forall y\in K\}.$$
Clearly, this is a pseudo-cone, and since ${\rm rec}\,K^*= ({\rm rec}\,K)^\circ=C^\circ$ by \cite[Lem. 4]{Sch24a}, it is a $C^\circ$-pseudo-cone. Therefore, the functions and maps defined above for $K$ are also well defined for $K^*$, with correspondingly changed domains. 

We have
\begin{equation}\label{2.5}
\rho_K(v)=\frac{1}{|h_{K^*}(v)|}\quad\mbox{for } v\in \Omega_C.
\end{equation}
For the proof, we refer to \cite[Thm. 3.10]{XLL23}.

According to \cite[Def. 4]{Sch24a}, a pair $(y,w)\in\R^n\times\R^n$ is a crucial pair of $K$ if $y\in\partial K$ and $w$ is an outer normal vector of $K$ at $y$, normalized so that $\langle y,w\rangle=-1$. By \cite[Lem. 6]{Sch24a} and $K^{**}=K$, we know that $(y,w)$ is a crucial pair of $K$ if and only if $(w,y)$ is a crucial pair of $K^*$. From this, it follows easily that, for $K\in psi(C)$, 
\begin{equation}\label{2.4} 
\alpha_K^*(u)= \alpha_{K^*}(u)  \quad\mbox{for }u\in \Omega_{C^\circ}\setminus\omega_K.
\end{equation}

From (\ref{2.5}), we can deduce the following. If $K_j\in ps(C)$ for $j\in\N_0$, then
$$
K_j\to K_0 \Leftrightarrow K_j^*\to K_0^*.
$$
Using (\ref{2.5}), we can also  prove the following lemma.

\begin{lemma}\label{L2.1} 
Let $K_j\in psi(C)$ for $j\in\N_0$. If $K_j\to K_0$ as $j\to \infty$, then
$$ \alpha_{K_j}^*\to \alpha_{K_0}^*$$
almost everywhere on $\Omega_{C^\circ}$, with respect to spherical Lebesgue measure.
\end{lemma}

\begin{proof}
The arguments are so similar to those used in the proof of (3.8) in \cite{BLYZZ20} that we need not repeat them here. The uniform convergence appearing in the employed Lemma 2.2 of \cite{HLYZ16} can be replaced by uniform convergence on compact sets.
\end{proof}

Wulff shapes in cones were introduced in \cite{Sch18}, and we repeat the definition here. Given a nonempty compact set $\omega\subset\Omega_{C^\circ}$ and a continuous function $f:\omega\to (0,\infty)$, let
$$ [f]:= C\cap\bigcap_{u\in\omega} \{y\in\R^n: \langle y,u\rangle \le -f(u)\}.$$
Then $[f]\in\K(C,\omega)$, and $[f]$ is called the Wulff shape associated with $(C,\omega,f)$. We have
$$ \overline h_{[f]}(u)\ge f(u)\quad\mbox{for }u\in\omega,$$
as follows immediately from the definition. According to \cite[(15)]{Sch24}, the radial function of the Wulff shape is given by
\begin{equation}\label{4.0} 
\frac{1}{\rho_{[f]}(v)} = \min_{u\in\omega} \frac{|\langle v,u\rangle|}{f(u)}\quad\mbox{for }v\in\Omega_C.
\end{equation}

We modify a definition in \cite{HLYZ16} (see also \cite[Sect. 4]{LYZ24}). Given a nonempty compact set $\eta\subset\Omega_C$ and a continuous function $f:\eta\to(0,\infty)$, we define 
$$ \langle f\rangle :=\bigcap \{K\in ps(C): f(v)v\in K\;\forall v\in\eta\}.$$
Since $f$ is positive, there always exists $K\in ps(C)$ with $f(v)v\in K$ for all $v\in\eta$. Clearly, $\langle f\rangle$ is a $C$-pseudo-cone. It is contained in ${\rm int}\,C$, so that, in particular, $\langle f\rangle\in psi(C)$. We call $\langle f\rangle$ the {\em convexification} associated with $(C,\eta,f)$. The definition implies that
\begin{equation}\label{4.1}
\rho_{\langle f\rangle}(v) \le f(v)\quad\mbox{for }v\in\eta.
\end{equation}
For the support function of the convexification we have 
\begin{equation}\label{4.2}
\overline h_{\langle f\rangle}(u) = \min_{v\in\eta} |\langle u,v\rangle|f(v)\quad\mbox{for }u\in{\rm cl}\,\Omega_{C^\circ}.
\end{equation}
In fact, from (\ref{2.1}) and (\ref{4.1}) it follows that
$$ \overline h_{\langle f\rangle}(u) = \inf_{v\in\Omega_C}|\langle u,v\rangle|\rho_{\langle f\rangle}(v) \le \min_{v\in\eta}|\langle u,v\rangle|\rho_{\langle f\rangle}(v) \le \min_{v\in\eta}|\langle u,v\rangle|f(v).$$
On the other hand, if $\min_{v\in\eta}|\langle u,v\rangle|f(v)=:t$ then
$$ C\cap\{y\in\R^n: |\langle y,u\rangle|\ge t\}\in ps(C)$$
and hence $\overline h_{\langle f\rangle}(u)\ge t$.

We point out that, in the definition of $\langle f\rangle$, the function $f$ is defined on a compact subset of $\Omega_C$, whereas in the case of the Wulff shape it is defined on a compact subset of $\Omega_{C^\circ}$. Of course, there is also the Wulff shape $[f]$ for $f$ defined on a compact subset of $\Omega_C$; in this case, $[f]$ is a $C^\circ$-pseudo-cone. Analogously, for a function $f$ on a compact subset $\omega\subset\Omega_{C^\circ}$, the convexification associated with $(C^\circ,\omega,f)$, denoted also by $\langle f\rangle$, is a $C^\circ$-pseudo-cone. It is not necessary to use different notations, since the domain of $f$ determines whether the result is in $ps(C)$ or in $ps(C^\circ)$.

Now we can state the relation
\begin{equation}\label{4.3}
[f]^*=\langle 1/f\rangle
\end{equation}
for a continuous function $f:\eta\to(0,\infty)$ defined on a nonempty compact set $\eta\subset\Omega_C$. In fact, for $u\in \Omega_{C^\circ}$ we have by (\ref{4.0}), (\ref{4.2}), (\ref{2.5}) and $K^{**}=K$ that
$$
\rho^{-1}_{[f]}(u)= \min_{v\in\eta} |\langle u,v\rangle|f^{-1}(v) = \overline h_{\langle f^{-1}\rangle}(u)= \rho^{-1}_{\langle f^{-1}\rangle^*}(u)
$$
and thus $[f]=\langle f^{-1}\rangle^*$, hence $[f]^*=\langle f^{-1}\rangle$.

We use (\ref{4.3}) to show the following continuity property of the convexification.

\begin{lemma}\label{L4.a}
Let $\eta\subset\Omega_C$ be a nonempty compact set, and let $f_j$ be a positive, continuous function on $\eta$, for $j\in\N_0$. If the sequence $(f_j)_{j\in\N}$ converges uniformly on $\eta$ to $f_0$, then $\langle f_j\rangle\to \langle f_0\rangle$ as $j\to\infty$.
\end{lemma}

\begin{proof} 
If $f_j\to f_0$ uniformly on $\eta$, where $f_0$ has a positive minimum, then also $f_j^{-1}\to f_0^{-1}$ uniformly on $\eta$. By Lemma 5 of \cite{Sch18} the Wulff shapes satisfy $[f_j^{-1}]\to [f_0^{-1}]$. By (\ref{4.3}) and (\ref{2.5}) this implies the assertion.
\end{proof}

Given a nonempty compact subset $\omega$ of $\Omega_{C^\circ}$, we have defined the $C$-pseudo-cones that are $C$-determined by $\omega$. Correspondingly, given a nonempty compact subset $\eta$ of $\Omega_C$ we now say that $K\in ps(C)$ is {\em $C$-defined by} $\eta$ if $K=\langle \rho_K|_\eta\rangle$, that is, if $K$ is the convexification associated with $(C,\eta,\rho_K|_\eta)$, where $\rho_K|_\eta$ is the restriction of the radial function of $K$ to $\eta$. We remark that also
$$ \langle \rho_K|_\eta\rangle = r_K(\eta)+C,$$
so that $\langle \rho_K|_\eta\rangle\in psi(K)$. The set of all $C$-pseudo-cones that are $C$-defined by $\eta$ is denoted by $\K^*(C,\eta)$. We have
\begin{equation}\label{4.3b}
K\in\K^*(C,\eta)\Leftrightarrow K^*\in \K(C^\circ,\eta).
\end{equation}
This can be proved along the following lines, using \cite[Lem. 6]{Sch24a}:
\begin{eqnarray*}
&& K\in \K^*(C,\eta)\\
&& \Leftrightarrow \mbox{if $(y,w)$ is a crucial pair of $K$ with $y/\|y\|\notin\eta$, then $w\in\partial C^\circ$}\\
&& \Leftrightarrow \mbox{if $(w,y)$ is a crucial pair of $K^*$ with $y/\|y\|\notin\eta$, then $w\in\partial C^\circ$}\\
&& \Leftrightarrow K^*\in \K(C^\circ,\eta).
\end{eqnarray*}

Concerning (\ref{4.3b}), we remark that (for compact $\eta\subset\Omega_C$) $K^*\in\K(C^\circ,\eta)$ is always $C^\circ$-full (that is $C^\circ\setminus K^*$ is bounded), but $K\in\K^*(C,\eta)$ is never $C$-full.

\section{About $\lambda(K,\cdot)$}\label{sec3}

In this section we assume that $K\in psi(C)$. Let $\lambda$ be a non-zero finite measure on the $\sigma$-algebra of Lebesgue measurable subsets of ${\rm cl}\,\Omega_{C^\circ}$ which vanishes on Borel sets of Hausdorff dimension $n-2$. As already mentioned, we define
$$ \lambda(K,\eta):= \lambda({\boldsymbol\alpha}_K(\eta)),\quad\eta\subseteq \Omega_C\mbox{ Borel}.$$
That ${\boldsymbol\alpha}_K(\eta)$ is indeed Lebesgue measurable, follows from the continuity of the radial map and \cite[Lem. 2.2.13]{Sch14}.

We state that, for $\eta_1,\eta_2\subseteq \Omega_C$,
$$ \eta_1\cap\eta_2=\emptyset\Rightarrow {\boldsymbol\alpha}_K(\eta_1)\cap {\boldsymbol\alpha}_K(\eta_2)\in\omega_K.$$
In fact, suppose that $\eta_1\cap\eta_2=\emptyset$ and let $u\in{\boldsymbol\alpha}_K(\eta_1)\cap {\boldsymbol\alpha}_K(\eta_2)$. Then $u\in{\boldsymbol \nu}_K(r_K(\eta_i))$ for $i=1,2$, hence $u$ is a normal vector of $K$ at a point $y_1\in r_K(\eta_1)$ and a point $y_2\in r_K(\eta_2)$. Since $r_K(\eta_1)\cap r_K(\eta_2)=\emptyset$, we have $y_1\not= y_2$ and thus $u\in \omega_K$. 

It follows that if $\eta_1,\eta_2,\dots$ is a sequence of pairwise disjoint Borel sets in $\Omega_C$, then the sets of the sequence ${\boldsymbol\alpha}_K(\eta_1), 
{\boldsymbol\alpha}_K(\eta_2)\dots$ are pairwise disjoint up to a set belonging to $\omega_K$. By (\ref{2.x}), $\lambda(\omega_K)=0$. Now it follows as in \cite{BLYZZ20} that $\lambda(K,\cdot)$ is a measure. Then, in particular, $\lambda(\partial\Omega_{C^\circ})=0$.

Since now $\lambda(K,\cdot)$ is a measure, it follows from (\ref{2.3}) that $\lambda(K,\cdot)$ is the image measure (or push-forward) of $\lambda$ under the continuous and hence measurable map $\alpha_K^*:\Omega_{\C^\circ}\setminus \omega_K\to \Omega_C$. We write this as
\begin{equation}\label{3.a}
\lambda(K,\cdot)= (\alpha_K^*)\texttt{\#}\lambda.
\end{equation}

Below we shall use tacitly the following lemma.

\begin{lemma}\label{4.0a}
If $\eta\subset\Omega_C$ is compact and $K\in\K^*(C,\eta)$, then $\lambda(K,\cdot)$ is concentrated on $\eta$.
\end{lemma}

\begin{proof}
Setting $\eta':=\Omega_C\setminus\eta$, any outer unit normal vector of $K$ at $r_K(v)$ with $v\in\eta'$ belongs to $\partial \Omega_{C^\circ}$. Thus, ${\boldsymbol\alpha}_K(\eta')\subset\partial \Omega_{C^\circ}$ and hence $\lambda({\boldsymbol\alpha}_K(\eta'))=0$.
\end{proof}

The following lemma corresponds to Lemma 3.3 in \cite{BLYZZ20}, but the approach is different.

\begin{lemma}\label{L3.1}
Let $K\in psi(C)$, and let $g:\Omega_C\to\R$ be a bounded Borel function. Then
\begin{equation}\label{3.1}
\int_{\Omega_{C^\circ}} g(\alpha_K^*(u))\,\lambda(\D u) = \int_{\Omega_C} g(v)\,\lambda(K,\D v).
\end{equation}
\end{lemma}

\begin{proof}
By (\ref{3.a}), the transformation theorem for integrals gives that, for each bounded Borel function $g:\Omega_C\to\R$,
$$ \int_{\Omega_C} g(v)\,\lambda(K,\D v) = \int_{\Omega_{C^\circ}\setminus\omega_K} g(\alpha_K^*(u))\,\lambda(\D u).$$
Since $\omega_K$ has $\lambda$-measure zero, we can also write
$$
\int_{\Omega_C} g(v)\,\lambda(K,\D v) = \int_{\Omega_{C^\circ}} g(\alpha_K^*(u))\,\lambda(\D u),
$$
as stated.
\end{proof}

A first consequence is the following weak continuity.

\begin{lemma}\label{L3.2}
Let $K_j\in psi(C)$ for $j\in\N_0$. Then $K_j\to K_0$ as $j\to \infty$ implies the weak convergence $\lambda(K_j,\cdot) \stackrel{w}{\to}  \lambda(K_0,\cdot)$.
\end{lemma}

\begin{proof}
Let $g:\Omega_C\to\R$ be continuous and bounded. By (\ref{3.1}) we have
$$ \int_{\Omega_C} g(v)\,\lambda(K_j,\D v) = \int_{\Omega_{C^\circ}} g(\alpha_{K_j}^*(u))\,\lambda(\D u).$$
The map $\alpha_{K_j}^*$ is continuous on its domain, and it is defined on $\Omega_{C^\circ}\setminus \omega_{K_j}$, where $\omega_{K_j}$ has spherical Lebesgue measure zero. Therefore, the functions $g\circ\alpha_{K_j}^*$ are measurable, and they are uniformly bounded. Now Lemma \ref{L3.1} and the dominated convergence theorem show that
$$ \int_{\Omega_C} g(v)\,\lambda(K_j,\D v) \to \int_{\Omega_C} g(v)\,\lambda(K_0,\D v) \quad\mbox{as }j\to\infty,$$
which gives the assertion.
\end{proof}

\section{Variational lemmas}\label{sec4}

The following lemma, which was proved by adapting arguments from \cite{HLYZ16}, is Lemma 5.4 in \cite{LYZ24} and also Lemma 9 in \cite{Sch24}. 

\begin{lemma}\label{L4.1}
Let $\omega\subset\Omega_{C^\circ}$ be nonempty and compact, and let $f_0:\omega\to (0,\infty)$ and $g:\omega\to \R$ be continuous. Define $f_t$ by 
$$ \log f_t(u)= \log f_0(u)+tg(u)\quad\mbox{for } u\in\omega,$$
for $|t|\le\delta$, for some $\delta>0$. Let $[f_t]$ be the Wulff shape associated with $(C,\omega,f_t)$. Then, for almost all $v\in\Omega_C$,
$$ \lim_{t\to 0} \frac{\log \rho_{[f_t]}(v) -\log\rho_{[f_0]}(v)}{t}= g(\alpha_{[f_0]}(v)). $$
\end{lemma}

(Note that $[f_0]\in\K(C,\omega)$, so that the outer unit normal vectors of $K$ at $r_K(\Omega_C)$ belong to $\omega$. Therefore, $\alpha_{[f_0]}(v)\in\omega$, so that $g(\alpha_{[f_0]}(v))$ is well defined for $v\in\Omega_C\setminus\eta_K$.)

We reformulate this lemma, interchanging the roles of $C$ and $C^\circ$.

\begin{lemma}\label{L4.2}
Let $\eta\subset\Omega_C$ be nonempty and compact, and let $f_0:\eta\to (0,\infty)$ and $g:\eta\to \R$ be continuous. Define $f_t$ by 
$$ \log f_t(v)= \log f_0(v)+tg(v)\quad\mbox{for } v\in\eta$$
for $|t|\le\delta$, for some $\delta>0$. Let $[f_t]$ be the Wulff shape associated with $(C^\circ,\eta,f_t)$. Then, for almost all $u\in\Omega_{C^\circ}$,
$$ \lim_{t\to 0} \frac{\log \rho_{[f_t]}(u) -\log\rho_{[f_0]}(u)}{t}= g(\alpha_{[f_0]}(u)). $$
\end{lemma}
(Note that the mapping $\alpha_{[f_0]}$, according to the domain of $f_0$, now goes from a part of $\Omega_{C^\circ}$ to $\Omega_C$.)

The following two lemmas correspond to Section 4 in \cite{BLYZZ20}.

\begin{lemma}\label{L4.3}
Let $\eta\subset\Omega_C$ be nonempty and compact, and let $f_0:\eta\to (0,\infty)$ and $g:\eta\to \R$ be continuous. Define $f_t$ by 
$$ \log f_t(v)= \log f_0(v)+tg(v)\quad\mbox{for } v\in\eta$$
for $|t|\le\delta$, for some $\delta>0$. Let $\langle f_t\rangle$ be the convexification associated with $(C,\eta,f_t)$.
 
\noindent $\rm (a)$ For almost all $u\in\Omega_{C^\circ}$,
\begin{equation}\label{4.4} 
\lim_{t\to 0} \frac{\log\overline h_{\langle f_t\rangle}(u) -\log \overline h_{\langle f_0\rangle}(u)}{t}=  g(\alpha^*_{\langle f_0\rangle}(u)). 
\end{equation}

\noindent $\rm (b)$ A constant $M$ and the constant $\delta>0$ can be chosen such that
$$ |\log\overline h_{\langle f_t\rangle}(u) -\log \overline h_{\langle f_0\rangle}|\le M|t|$$
for all $u\in \Omega_{C^\circ}$ and all $|t|\le\delta$.
\end{lemma}

\begin{proof}
For the proof of part (a), we `reverse' the arguments in the proof of Lemma 4.3 in \cite{HLYZ16}. We have
$$ \log f_t^{-1} = \log f_0^{-1}-tg,$$
hence Lemma \ref{L4.2} (applied to $f_t^{-1}$ and $-g$) gives
$$ \lim_{t\to 0} \frac{\log \rho_{[f_t^{-1}]}(u) -\log\rho_{[f_0^{-1}]}(u)}{t}= -g(\alpha_{[f_0^{-1}]}(u)) $$
for almost all $u\in\Omega_{C^\circ}$. 

Since $[f_t^{-1}]=\langle f_t\rangle^*$ by (\ref{4.3}) and $\rho_{K^*}= 1/\overline h_K$ by (\ref{2.5}) (using $K^{**}=K$), we get
$$ \log \rho_{[f_t^{-1}]}-\log \rho_{[f_0^{-1}]} = \log \rho_{\langle f_t\rangle^*}-\log \rho_{\langle f_0\rangle^*} =-\left( \log\overline h_{\langle f_t\rangle} -\log \overline h_{\langle f_0\rangle}\right).$$
Moreover, from (\ref{2.4}) (noting that $\langle f_0\rangle\in psi(C)$) we have
$$ \alpha_{[f_0^{-1}]}(u) = \alpha_{\langle f_0\rangle^*}(u) = \alpha^*_{\langle f_0\rangle}(u) \quad\mbox{for } u\in\Omega_{C^\circ}\setminus \omega_{\langle f_0\rangle}. $$
This completes the proof of (a).

For part (b), let $u\in\Omega_{C^\circ}$ be such that (\ref{4.4}) is satisfied. We abbreviate $F(t):= \log\overline h_{\langle f_t\rangle}(u)$ and have
$$\left|\frac{F(t)-F(0)}{t}\right|-\left|\frac{\D F(t)}{\D t}\Big|_{t=0}\right|\le \left|\frac{F(t)-F(0)}{t}-\frac{\D F(t)}{\D t}\Big|_{t=0}\right|\le 1 \quad\mbox{for }|t|\le \delta$$
by (a), if $\delta>0$ is chosen appropriately. It follows that
$$ \left|\frac{F(t)-F(0)}{t}\right|\le \left|\frac{\D F(t)}{\D t}\Big|_{t=0}\right|+1=g(\alpha^*_{\langle f_0\rangle}(u)) +1,$$
which is bounded by a constant $M$ depending only on $g$. By continuity, this holds for all $u\in \Omega_{C^\circ}$.
\end{proof}

The following lemma corresponds to the first part of Lemma 4.2 in \cite{BLYZZ20}. Note that we define $f_t$ only on $\eta$, so that $f_0=\rho_K|_\eta$, the restriction of the radial function of $K$ to $\eta$. Observe also that $\lambda(\langle f_0\rangle,\cdot)$ is concentrated on $\eta$.

\begin{lemma}\label{L4.4}
Let $\lambda$ be a measure on $\Omega_{C^\circ}$ as in Section $\rm \ref{sec3}$, let $K\in ps(C)$. Let $\eta\subset \Omega_C$ be nonempty and compact, and let $g:\eta\to\R$ be continuous. Define
$$ \log f_t(v) = \log \rho_K(v)+tg(v)\quad\mbox{for }v\in\eta$$
and let $\langle f_t\rangle$ be the convexification associated with $(C,\eta,f_t)$. Then
\begin{equation}\label{4.5}
 \frac{\D}{\D t} \int_{\Omega_{C^\circ}}\log \rho_{\langle f_t\rangle^*}(u)\,\lambda(\D u)\Big|_{t=0}= -\int_{\eta} g(v)\,\lambda(\langle f_0\rangle,\D v).
\end{equation}
\end{lemma}

\begin{proof}
We use (\ref{2.5}), the dominated convergence theorem, Lemma \ref{L4.3} and Lemma \ref{L3.1} (extending $g$ to $\Omega_C$ by putting it equal to $0$ outside $\eta$), observing that $\langle f_0\rangle\in psi(C)$ and the image of $\alpha^*_{\langle f_0\rangle}$ is contained in $\eta$. In this way, we obtain
\begin{eqnarray*}
 \frac{\D}{\D t} \int_{\Omega_{C^\circ}}\log \rho_{\langle f_t\rangle^*}(u)\,\lambda(\D u)\Big|_{t=0} &=& -\lim_{t\to 0} \int_{\Omega_{C^\circ}} \frac{\log\overline h_{\langle f_t\rangle}(u) -\log \overline h_{\langle f_0\rangle}(u)}{t}\,\lambda(\D u)\\
&=& - \int_{\Omega_{C^\circ}} g(\alpha_{\langle f_0\rangle}^*)\,\lambda(\D u)\\
&=& -\int_{\eta} g(v)\,\lambda(\langle f_0\rangle,\D v)
\end{eqnarray*}
and thus the assertion.
\end{proof}

Since $\rho_{K^*} = \overline h^{-1}_{K}$ by (\ref{2.5}), we can write (\ref{4.5}) also in the form
\begin{equation}\label{4.6}
 \frac{\D}{\D t} \int_{\Omega_{C^\circ}}\log \overline h_{\langle f_t\rangle}(u)\,\lambda(\D u)\Big|_{t=0}= \int_{\eta} g(v)\,\lambda(\langle f_0\rangle,\D v).
\end{equation}

These lemmas serve as a preparation for deriving a variational formula for a useful functional. For a finite Borel measure $\mu$ and the Lebesgue measure $\lambda$ on the sphere $\Sn$, Oliker \cite{Oli07} introduced the functional
$$ K\mapsto \int_{\Sn}\log\rho_K\,\D\mu - \int_{\Sn}\log h_K\,\D\lambda$$
for convex bodies $K$ containing the origin in the interior. In \cite{BLYZZ20}, this functional was extended to more general measures $\lambda$ and to functions as arguments. We adapt this definition here to pseudo-cones.

\begin{definition}\label{D4.1}
Let $\lambda$ be a non-zero, finite measure on the Lebesgue measurable subsets of $\Omega_{C^\circ}$ which vanishes on Borel sets of Hausdorff dimension $n-2$. Let $\eta\subset\Omega_C$ be nonempty and compact, let $\C^+(\eta)$ be the space of positive continuous functions on $\eta$, equipped with the maximum norm, and let $\mu$ be a non-zero finite Borel measure on $\eta$. Then define
$$ \Phi_{\mu,\lambda,\eta}(f):= \frac{1}{\mu(\eta)}\int_{\eta} \log f(v)\,\mu(\D v) - \frac{1}{|\lambda|}\int_{\Omega_{C^\circ}} \log \overline h_{\langle f\rangle}(u)\,\lambda(\D u)$$
for $f\in \C^+(\eta)$, where $|\lambda|:=\lambda(\Omega_{C^\circ})$.
\end{definition}

The functional $\Phi_{\mu,\lambda,\eta}$ is positively homogeneous of degree zero, as follows from $\overline h_{\langle af\rangle} = \overline h_{a\langle f\rangle} = a \overline h_{\langle f\rangle}$ for $a>0$, and continuous, by Lemma \ref{L4.a}.

\begin{lemma}\label{L4.6}
Let $\eta\subset\Omega_C$ be nonempty and compact. Let $K\in ps(C)$, and let $g:\eta\to\R$ be continuous. Define $f_t$ by 
$$ \log f_t(v)= \log\rho_{K}(v)+tg(v)\quad\mbox{for } v\in\eta$$
for $|t|\le\delta$, for some $\delta>0$. Then
$$ \frac{\D}{\D t} \Phi_{\mu,\lambda,\eta}(f_t)\Big|_{t=0} = \frac{1}{\mu(\eta)} \int_{\eta} g(v)\,\mu(\D v) -\frac{1}{|\lambda|}  \int_{\eta} g(v)\, \lambda(\langle f_0\rangle,\D v).$$
\end{lemma}

\begin{proof}
We have
$$ \int_{\eta} \log f_t(v)\,\mu(\D v) = \int_{\eta} (\log \rho_{K} +tg)(v)\,\mu(\D v),$$
hence
$$ \frac{\D}{\D t} \int_{\eta} \log f_t(v)\,\mu(\D v) = \int_{\eta} g(v)\,\mu(\D v).$$
Together with (\ref{4.6}) this gives the assertion.
\end{proof}

\section{Proof of existence}\label{sec5}

We assume that $\lambda,\eta,\mu$ are as in Definition \ref{D4.1}.

\begin{lemma}\label{L5.1}
There exists $K_0\in ps(C)$ such that
$$ \Phi_{\mu,\lambda,\eta}(\rho_{K_0}|_\eta) \le \Phi_{\mu,\lambda,\eta}(\rho_K|_\eta) \quad\mbox{for all } K\in ps(C).$$
\end{lemma}

\begin{proof}
Let $(K_j)_{j\in\N}$ be a sequence of pseudo-cones satisfying
$$ \Phi_{\mu,\lambda,\eta}(\rho_{K_j}|_\eta) \to \inf_{K\in ps(C)} \Phi_{\mu,\lambda,\eta}(\rho_{K}|_\eta)$$
(note that $\Phi_{\mu,\lambda,\eta}>0$). Since $ \Phi_{\mu,\lambda,\eta}$ is homogeneous of degree zero, we can assume that each $K_j$ has distance $1$ from the origin. By Lemma 1 in \cite{Sch24a}, the sequence $(K_j)_{j\in\N}$ has a subsequence converging to a $C$-pseudo-cone $K_0\in ps(C)$. The continuity of $\Phi_{\mu,\lambda,\eta}$ now gives the assertion.
\end{proof}

From this, we can deduce the following.

\begin{lemma}\label{L5.2}
With $\lambda,\eta,\mu$ as given above, there exists $K\in \K^*(C,\eta)$, having distance $1$ from the origin, such that
$$ \mu= \frac{\mu(\eta)}{|\lambda|}\lambda(K,\cdot).$$
\end{lemma}

\begin{proof}
First, let $f\in \C^+(\eta)$. We have $f\ge \rho_{\langle f\rangle}|_\eta$ by (\ref{4.1}), moreover $\langle\rho_{\langle f\rangle}|_\eta\rangle =\langle f\rangle$, hence $\overline h_{\langle\rho_{\langle f\rangle}|_\eta\rangle} = \overline h_{\langle f\rangle}$. It follows that 
\begin{eqnarray*}
 \Phi_{\mu,\lambda,\eta}(f) &=& \frac{1}{\mu(\eta)}\int_{\eta} \log f(v)\,\mu(\D v) - \frac{1}{|\lambda|}\int_{\Omega_{C^\circ}} \log \overline h_{\langle f\rangle}(u)\,\lambda(\D u)\\
&\ge& \frac{1}{\mu(\eta)}\int_{\eta} \log \rho_{\langle f\rangle}|_\eta(v)\,\mu(\D v) - \frac{1}{|\lambda|}\int_{\Omega_{C^\circ}} \log \overline h_{\langle \rho_{\langle f\rangle}|_\eta\rangle}(u)\,\lambda(\D u)\\
&=& \Phi_{\mu,\lambda,\eta}( \rho_{\langle f\rangle}|_\eta)\\
&\ge& \Phi_{\mu,\lambda,\eta}( \rho_{K_0}|_\eta),
\end{eqnarray*}
if $K_0$ is the pseudo-cone provided by Lemma \ref{L5.1}. Let $g$ be a continuous function on $\eta$, and define $f_t$ by
$$ \log f_t(v)=\log \rho_{K_0}(v) + tg(v)\quad\mbox{for } v\in\eta.$$
Then the function $t\mapsto\Phi_{\mu,\lambda,\eta}(f_t)$ attains a minimum at $t=0$. Therefore, it follows from Lemma \ref{L4.6} that
$$ \frac{1}{\mu(\eta)} \int_{\eta} g(v)\,\mu(\D v) -\frac{1}{|\lambda|}  \int_{\eta} g(v)\, \lambda(\langle f_0\rangle,\D v)=0.$$
Since this holds for all continuous functions $g$ on $\eta$, we can conclude that
\begin{equation}\label{5.a} 
\mu= \frac{\mu(\eta)}{|\lambda|}\lambda(K,\cdot)
\end{equation}
with $K:= \langle f_0\rangle = \langle \rho_{K_0}|_\eta\rangle$ (note that $K$ depends on $\lambda,\eta,\mu$). Since $\lambda(K,\cdot)$ is invariant under dilatations of $K$, we can assume that $K$ has distance $1$ from the origin.
\end{proof}

Now we can prove our main result.

\noindent{\em Proof of Theorem $\ref{T1.1}$}

Let $\lambda$ and $\mu$ be as in Theorem \ref{T1.1}. That the condition $\lambda(\Omega_{C^\circ})=\mu(\Omega_C)$ is necessary, is clear. Now we assume that this condition is satisfied.

We choose a sequence $(\eta_j)_{j\in\N}$ of compact sets $\eta_j\subset\Omega_C$ such that
$$ \mu(\eta_1)>0,\qquad \eta_j\subset \eta_{j+1} \mbox{ for }j\in\N,\qquad \bigcup_{j\in\N} \eta_j= \Omega_C.$$
Define the measure $\mu_j= \mu\fed\eta_j$, that is, 
$$ \mu_j(\beta)= \mu(\beta\cap \eta_j) \quad\mbox{for Borel sets }\beta\subset \Omega_C.$$
By Lemma \ref{L5.2}, to each $j\in \N$ there exists $K_j\in\K^*(C,\eta_j)\subset ps(C)$, having distance $1$ from the origin, such that
$$ \mu_j= \frac{\mu(\eta_j)}{|\lambda|}\lambda(K_j,\cdot).$$
By Lemma 1 of \cite{Sch24a}, the sequence $(K_j)_{j\in\N}$ has a subsequence that converges to a $C$-pseudo-cone $K$. After changing the numbering, we may assume that $K_j\to K$ as $j\to\infty$. 

Fix a number $k\in\N$. For a Borel set $\beta\subseteq \eta_k$ and for $j\ge k$ we have
$$ \mu(\beta)= \mu_j(\beta) = \frac{\mu(\eta_j)}{|\lambda|}\lambda(K_j,\beta).$$
Thus, the restrictions to $\eta_k$ satisfy
$$ \frac{|\lambda|}{\mu(\eta_j)}\mu\fed \eta_k = \lambda(K_j,\cdot)\fed\eta_k.$$
As $j\to\infty$, we have $|\lambda|/\mu(\eta_j)\to 1$ and
$$ \lambda(K_j,\cdot) \fed\eta_k \stackrel{w}{\to} \lambda(K,\cdot)\fed \eta_k,$$
as follows from Lemma \ref{L3.2}. Thus for any Borel set $\beta\subseteq\eta_k$ we have $\mu(\beta)=\lambda(K,\beta)$. Since $k\in\N$ was arbitrary and $\bigcup_{k\in\N} \eta_k=\Omega_C$, we conclude that $\mu=\lambda(K,\cdot)$. \hfill$\Box$

According to (\ref{3.a}), we can write the result also in the form
$$ \mu = (\alpha_K^*)\texttt{\#}\lambda.$$

If $\lambda$ is spherical Lebesgue measure, Theorem \ref{T1.1} concerns Aleksandrov's integral curvature. For measures $\mu$ with compact support, such a result was already obtained in \cite[Thm. 8.3]{LYZ24}.

\section{A uniqueness result}\label{sec6}

The uniqueness result provided by Lemma 3.8 of \cite{BLYZZ20} can be carried over to pseudo-cones, under suitable assumptions. Originally, in the case of Aleksandrov's integral curvature, the approach is due to Aleksandrov \cite{Ale42}; see also Busemann \cite[p. 30]{Bus58}. Aleksandrov's argument has repeatedly been adapted to other situations, so we claim no originality when we now carry over the proof of Lemma 3.8 in \cite{BLYZZ20} from convex bodies to $C$-pseudo-cones. We just want to point out the necessary changes. We have to assume that the considered $C$-pseudo-cones are {\em restricted}, in the sense that all their outer unit normal vectors belong to $\Omega_{C^\circ}$. This excludes all $C$-pseudo-cones $K$ meeting the boundary of $C$, as well as, for example, the pseudo-cones $C+K$ with a convex body $K\subset{\rm int}\,C$.

\begin{theorem}\label{T6.1}
Suppose that the measures $\lambda$ and $\mu$ are as in Theorem $\ref{T1.1}$, and that $\lambda$ is positive on nonempty open sets in $\Omega_{C^\circ}$. If $K,L\in ps(C)$ are restricted pseudo-cones with $\lambda(K,\cdot)=\lambda(L,\cdot)$, then $K$ is a dilate of $L$.
\end{theorem}

\begin{proof}
Suppose that $K,L\in ps(C)$ are restricted and satisfy $\lambda(K,\cdot)=\lambda(L,\cdot)$. We assume that there exist a dilate $K'$ of $K$   and a vector $v_0\in \Omega_C$ such that $y_0:= r_{K'}(v_0)= r_L(v_0)$ and such that $K'$ and $L$ have unique supporting hyperplanes at $y_0$, which are different. As in \cite{BLYZZ20}, it suffices to show that this leads to a contradiction. 

We change the definitions of \cite{BLYZZ20} as follows. We set
\begin{eqnarray*}
\eta' &=& \{v\in \Omega_C: \rho_{K'}(v)<\rho_L(v)\},\\
\eta &=& \{v\in\Omega_C: \rho_{K'}(v)>\rho_L(v)\},\\
\eta_0 &=& \{v\in\Omega_C: \rho_{K'}(v)=\rho_L(v)\},
\end{eqnarray*}
so that $\Omega_C= \eta'\cup \eta\cup \eta_0$ is a disjoint decomposition. Also the subsequent argument must be changed a bit. Let $v\in\eta'$, and let $H_L$ be a supporting hyperplane of $L$ at $r_L(v)$. Since the normal vector of $H_L$ belongs to $\Omega_{C^\circ}$, the pseudo-cone $K'$ has a parallel supporting hyperplane $H_{K'}$, at some point $r_{K'}(v')$ with $v'\in\Omega_C$. Then $H_{K'}$ is closer to $o$ than $H_L$, since otherwise the half-open segment $[o,r_L(v))$ would not contain a point of $K'$, which would mean that $v\notin\eta'$, a contradiction. Since $H_{K'}$ is closer to $o$ than $H_L$ and $r_{K'}(v')\in H_{K'}$, we have $r_{K'}(v')<r_L(v')$ and thus $v'\in\eta'$. We have proved that 
\begin{equation}\label{6.1}
{\boldsymbol\alpha}_L(\eta')\subseteq {\boldsymbol\alpha}_{K'}(\eta') = {\boldsymbol\alpha}_K(\eta')
\end{equation}

As in \cite{BLYZZ20}, one shows that the sets $\eta\cup \eta_0$ and $\eta'\cup\eta_0$ are closed in $\Omega_C$ and hence that $\Omega_{C^\circ}\setminus{\boldsymbol\alpha}_{K'}(\eta\cup\eta_0)$ and $\Omega_{C^\circ}\setminus{\boldsymbol\alpha}_L(\eta'\cup\eta_0)$ are open. Further,
$$
\Omega_{C^\circ}\setminus{\boldsymbol\alpha}_{K'}(\eta\cup\eta_0)\subset {\boldsymbol\alpha}_{K'}(\eta')
$$
and 
$$
(\Omega_{C^\circ}\setminus{\boldsymbol\alpha}_L(\eta'\cup\eta_0))\cap {\boldsymbol\alpha}_L(\eta')=\emptyset.
$$
The set
$$ \beta:= (\Omega_{C^\circ}\setminus{\boldsymbol\alpha}_{K'}(\eta\cup\eta_0)) \cap  (\Omega_{C^\circ}\setminus{\boldsymbol\alpha}_L(\eta'\cup\eta_0))$$
is open and
\begin{equation}\label{6.5} 
\beta\cap {\boldsymbol\alpha}_L(\eta')=\emptyset\quad\mbox{and}\quad \beta\subset {\boldsymbol\alpha}_{K'}(\eta').
\end{equation}

As in \cite{BLYZZ20}, one obtains that $\beta$ is not empty, hence $\lambda(\beta)>0$. From (\ref{6.1}) and (\ref{6.5}) we have
$$ {\boldsymbol\alpha}_L(\eta') = {\boldsymbol\alpha}_L(\eta'\setminus\beta)\subset {\boldsymbol\alpha}_{K'}(\eta')\setminus \beta$$
and hence
\begin{eqnarray*}
\lambda(L,\eta') =\lambda({\boldsymbol\alpha}_L(\eta')) &\le& \lambda({\boldsymbol\alpha}_{K'}(\eta')\setminus\beta)\\
&<& \lambda({\boldsymbol\alpha}_{K'}(\eta')\setminus\beta) +\lambda(\beta) = \lambda({\boldsymbol\alpha}_{K'}(\eta')) = \lambda(K,\eta'),
\end{eqnarray*}
a contradiction.
\end{proof}

\noindent Author's address:\\[2mm]
Rolf Schneider\\Mathematisches Institut, Albert--Ludwigs-Universit{\"a}t\\D-79104 Freiburg i.~Br., Germany\\E-mail: rolf.schneider@math.uni-freiburg.de

\end{document}